\documentclass[reqno]{amsart}
\usepackage{amsfonts,amssymb,amsbsy,amsmath,amsthm,enumerate}

\newtheorem{theorem}{Theorem}[section]
\newtheorem{proposition}[theorem]{Proposition}
\newtheorem{lemma}[theorem]{Lemma}
\newtheorem{corollary}[theorem]{Corollary}

\newtheorem{definition}[theorem]{Definition}
\newtheorem{remark}[theorem]{Remark}

\numberwithin{equation}{section}

\newcommand\beq{\begin{equation}}
\newcommand\eeq{\end{equation}}
\newcommand{\bea}{\begin{eqnarray}}
\newcommand{\eea}{\end{eqnarray}}
\newcommand{\beas}{\begin{eqnarray*}}
\newcommand{\eeas}{\end{eqnarray*}}
\newcommand{\bel}{\begin{equation} \label}
\newcommand{\ee}{\end{equation}}

\newcommand{\Q}{\mathcal{Q}}
\newcommand{\K}{\mathcal{K}}

\newcommand{\W}{\mathcal{W}}

\newcommand{\rd}{{\mathbb R}^{2}}
\newcommand{\re}{{\mathbb R}}
\newcommand{\Z}{\mathbb Z}

\newcommand{\al}{a_{\alpha}}

\newcommand{\be}{a_{\beta}}

\newcommand\cA{\mathcal A}

\newcommand\R{\mathbb R}
\newcommand\N{\mathbb N}

\newcommand\T{\mathcal{T}}
\def\N{\mathbb{N}}

\def\Z{\mathbb{Z}}

\newcommand\G{\mathbb{G}}

\newcommand{\1}{\mathbf{1}}
\renewcommand\div{\mathrm{div}}

\DeclareMathOperator{\supp}{supp}
\DeclareMathOperator{\tr}{tr}

\newcommand{\la}{\langle}
\newcommand{\ra}{\rangle}

\newcommand{\scal}[1]{\la #1 \ra}

\numberwithin{equation}{section}

\begin{document}

\title[Quantization of the edge currents for magnetic models]{Quantization of edge currents along magnetic barriers and magnetic guides}

\author[N. Dombrowski]{Nicolas Dombrowski}
\address[Dombrowski]{Universit\'e de Cergy-Pontoise,
CNRS UMR 8088, D\'epartement de Ma\-th\'e\-matiques, F-95000
Cergy-Pontoise, France; {\em Present address}: Pontificia
Universidad Cat\'olica de Chile, Facultad de Matem\'aticas, Av.
Vicu\~na Mackenna 4860,  Santiago de Chile}
\email{nicolas.dombrowski@u-cergy.fr}

\author[F. Germinet]{Fran\c cois Germinet}
\address[Germinet]{Universit\'e de Cergy-Pontoise,
CNRS UMR 8088, IUF, D\'epartement de Math\'ematiques,
F-95000 Cergy-Pontoise, France}
\email{germinet@math.u-cergy.fr}

\author[G. Raikov]{Georgi Raikov}
\address[Raikov]{Pontificia Universidad Cat\'olica de Chile, Facultad de Matem\'aticas,
Santiago de Chile}
\email{graikov@mat.puc.cl}



\begin{abstract}
 {We investigate the edge conductance of particles submitted to an Iwatsuka magnetic field, playing the role of a purely magnetic barrier. We also consider magnetic guides generated by generalized Iwatsuka potentials. In both cases we prove quantization of the edge conductance.  Next, we consider magnetic
perturbations of such magnetic barriers or guides, and prove stability of the quantized value of the edge
conductance. Further, we establish a sum rule for edge
conductances. Regularization within the context of disordered systems is discussed as well.}
\end{abstract}

\maketitle

 \today
\tableofcontents

\section{Introduction}
\label{sectintro}

Quantization of two dimensional edge states appeared within the
context of the Quantum Hall Effect in a seminal paper of Halperin
\cite{Ha}. Existence of quantum states flowing along edges has
been studied mathematically in several recent works, e.g.
\cite{DBP,FGW,FM1,FM2,CHS}. Besides the existence of such states,
the question of their quantization has been brought forth mathematically in
\cite{SBKR,EG,EGS,CG} through the quantization of a so called edge conductance,  together with the issue of the equality between that edge conductance and the bulk conductance, also called the Hall conductance (concerning the quantization of the bulk conductance itself, see \cite{Be,BES,ASS,AG,GKS1,GKS2}). It is worth pointing out that in all these works, the edge is modeled by
a confining electric potential or by a hard wall with Dirichlet
boundary condition, a case that could be interpreted as an
infinite electric potential in a half-plane. Considered perturbations are designed by electric potentials as well. In this article we are interested in the same phenomena, but generated by purely  magnetic effects. Furthermore, we investigate the quantization of currents carried by magnetic wave guides, a new feature compared to the electric case.

First, the wall is designed by an Iwatsuka magnetic field \cite{I}, i.e. a $y$-independent magnetic field with a decaying profile in the $x$-axis and of constant sign. As a matter of fact, the particle is subjected to, say, a strong magnetic field on the right half plane, and to a weaker one on the left half plane. The extreme version of this model would be a magnetic field with strength $B_->0$ for $x<0$ and $B_+>0$ for $x>0$, with $B_+ - B_-$ large enough. Due to the strength difference, the particle is localized on circle of radius $1/\sqrt{B_+}$ inside the left half plane and $1/\sqrt{B_-}>1/\sqrt{B_+}$ inside the right half plane; it is easy to be convinced that near the interface $x=0$, states with energies between $B_-$ and $B_+$ are extended, and induce a current flowing in the $y$ direction (see for instance Fig.~6.1 in \cite{CFKS}); so that the interface $x=0$ plays here  the role of a ```magnetic wall".

The spectral interpretation of this fact is the ac nature of the
spectrum, as proven in  \cite{I}. However ac spectrum does not shed
light on existing currents flowing along the edge. We shall
provide a simple computation of the edge conductance that
validates this intuition, showing it is non zero when considering
energies above the first Landau level of a Landau Hamiltonian of
magnetic intensity $B_-$; the edge conductance is actually
quantized, in concordance with the physics of the quantum effect
and Halperin's argument. We mention here that if the existence of
edge states for the half-plane model has been proved in case of
electric perturbations \cite{DBP,FGW}, by showing that band
functions of the unperturbed system  have a strictly positive
derivative, and then using Mourre's theory, a similar approach
fails with the Iwatsuka potential for band functions may not be
monotone. Nevertheless, our analysis goes through, because the
edge conductance ``computes" the net current, even if there are
several channels of opposite sign.

A totally different situation is the one where the magnetic field profile  is a  $y$-independent monotone decaying function of non constant sign. We will call such an Hamiltonian a generalized Iwatsuka Hamiltonian. The properties of such operators are nevertheless quite different from those of the standard Iwatsuka Hamiltonian. In particular, particles are confined to a strip rather than to a half plane. We mention that there has been some recent attention in the physcis literature for such quantum magnetic guides, for they exhibit interesting extended states called  ``snake-orbit" states \cite{RMCPV,RP}. We will prove that currents carried by such states are quantized as well.

Secondly, the perturbations we consider are also of magnetic nature. To motivate the study of such magnetic perturbations, let us recall  that relevant perturbations within the context of the quantum Hall effect are random perturbations (modeling impurities or defects of the sample), for the localized states they generate are responsible for the celebrated plateaux of the (bulk) Hall conductance. Occurrence of localized states which could arise from random magnetic perturbations in dimension 2 in relation with quantum Hall systems, has been intensively studied in the physics literature  over the past two decades  (see e.g. \cite{AHK,BSK,Fu,V}). Mathematically, the proof of the occurrence of Anderson localization due to random magnetic potentials only is not an easy task, and very few preliminary results are available: recently Ghribi, Hislop, and Klopp \cite{GHK} proved localization for  random magnetic perturbations of a periodic magnetic potential  (see also \cite{KNNY} for a particular discrete model). Ueki \cite{U} proved localization for some magnetic perturbation of the Landau Hamiltonian at the bottom of the spectrum (below the first Landau level). In the companion note \cite{DGR2}, we provide an example that is relevant to the theory of the quantum Hall effect, namely a  random magnetic perturbation of the Landau Hamiltonian with localized states at the edges of the first $J$ bands, $J\ge 1$ given.

As a preliminary result, we show that currents generated by Iwatsuka and generalized Iwatsuka Hamiltonians are quantized and we compute the exact value of the edge conductance. Next, we prove that magnetic perturbations carried by magnetic fields compactly supported in the $x$ axis do not affect the edge conductance. In particular, if we consider a magnetic strip and a moderate magnetic field inside, then the net current flowing along these axes is zero, like in the electric case. Then, we consider  perturbations which do not vanish at infinity, and provide a sum rule similar to the one obtained in \cite{CG}. Namely, the edge conductance of the perturbed system is the sum of the edge conductance of the magnetic confining potential and of the edge conductance of the system without magnetic wall defined by the Landau Hamiltonian of magnetic strength $B_-$ perturbed by the magnetic potential. This enables us to compute the edge conductance for the magnetically perturbed Hamiltonian when  energies  fall inside a gap of the unperturbed Landau Hamiltonian of magnetic strength $B_-$. To consider energies corresponding to localized states, one has to go one step further and regularize the trace that defines the edge conductance (see \cite{CG,EGS}) and use the localization properties of the model as provided by the theory of random Schr\"odinger operators \cite{AENSS,GK1}. As an illustration, we revisit the model considered in \cite{DGR2} and discuss the quantization of its  (regularized) edge conductance in presence of an Iwatsuka confining wall.

Of course, it follows from \cite{CG} and the results of the present paper that any combination of electric and magnetic potentials defining the confining wall and the perturbation will work in the same way.

The paper is organized as follows. In Section~\ref{sectmain} we state our main results. In Section~\ref{sectIw} we gather properties of the generalized Iwatsuka Hamiltonian and prove the quantization of the edge conductance in the unperturbed case. In Section~\ref{sectcompact} we consider compact (in the $x$ axis) magnetic
perturbations of magnetic barriers, and prove stability of the quantized value of the edge
conductance. In Section~\ref{sectSumrule},  we consider non vanishing magnetic perturbations and establish a sum rule for edge
conductances. In Section~\ref{sectregul} we discuss regularization of the edge conductance in the presence of disorder. Finally, in the Appendix (Section~\ref{sectAuxiliary}) we gather for  reader's convenience some trace estimates  used intensively in the main text.

\section{Definitions and Results}
\label{sectmain}

 \subsection{Edge conductance}

Let $A \in L^2_{\rm loc}(\re^2;\rd)$. Define
    $$
    H(A) : = (i\nabla + A)^2
    $$
    as the self-adjoint operator generated in $L^2(\rd)$ by the closure of the quadratic form
    $$
    \int_{\rd} |i\nabla u + A u|^2 dx, \quad u \in {\mathcal C}_0^{\infty}(\rd).
    $$
    If $A \in L^4_{\rm loc}(\rd; \rd)$ and ${\rm div}\,A \in L^2_{\rm loc}(\rd;\re)$, then $H(A)$ is essentially self-adjoint on ${\mathcal C}_0^{\infty}(\rd)$
    (see \cite{LS}).

    We will say that the magnetic potential $A = (A_1, A_2)$
generates the magnetic field $B : \re^2 \to \re$ if
    \bel{3}
    \frac{\partial A_2}{\partial x} - \frac{\partial A_1}{\partial
    y} = B(x,y), \quad (x,y) \in \rd.
    \ee

The celebrated Landau Hamiltonian corresponds to the case $B(x,y)=B\neq 0$ constant, in which case the spectrum is consists in the so called Landau levels $(2n+1)|B|$, $n\in \N : = \{0,1,2,\ldots\}$. To fix notations, we will denote by $H_B$ the Landau Hamiltonian, and set
\bel{defgap}
\G_0(B)=]-\infty, |B|[, \mbox{ and } \mathbb{G}_n(B)=](2n-1)|B|, (2n+1)|B|[  \mbox{ for } n\ge 1,
\ee
the (open) Landau gaps.

 We will  say that $f : \re \to [0,1]$ is {\em
an  increasing (resp. decreasing) switch function} if $f \in C^{\infty}(\re)$ is monotone, with a
compactly  supported derivative,  $f\equiv 1$, (resp., $f\equiv
0$) on the right side of $\supp f'$, and  $f\equiv 0$, (resp., $f\equiv
1$)  on the left.

\begin{definition}\label{defconduc}
Let $\chi\in \mathcal{C}^{\infty}(\re^2)$ be a x-translation
invariant increasing switch function with ${\rm supp}\,\chi'\subset
\re\times[-\frac{1}{4},\frac{1}{4}]$, and $ g\in
\mathcal{C}^{\infty}(\re)$ be a decreasing  switch function with $ {\rm
supp}\, g'\subset I=[a,b]$, a compact interval. The edge
conductance of a Hamiltonian $H$ in the interval $I$ is defined as
\begin{equation} \label{conduc}
 \sigma_e^{(I)}(H):= -2\pi \tr\,(g'(H)i[H,\chi]),
\end{equation}
whenever the trace exists.
\end{definition}

We note that although the above definition depends a priori on the
choice of $g$ and $\chi$, results will not.

Note also that when the magnetic field is reversed $B \mapsto -B$, then the edge conductance is changed into its opposite, whenever it exists.

The factor $2\pi$ is introduced in order that conductance  $\sigma_e^{(I)}(H)$ be integer-valued.

 \subsection{Generalized Iwatsuka Hamiltonians}

A magnetic field will be called a generalized Iwatsuka magnetic field if the following conditions hold:
\begin{itemize}
\item[{\bf GIw.1}] $B$ depends only on  the first coordinate, i.e. $B(x,y) = B(x)$;
\item[{\bf GIw.2}] $B$ is a monotone function of $x$;
\item[{\bf GIw.3}] There exist two numbers $B_-,B_+\in \R\setminus\{0\}$ such that
    $$
    \lim_{x \to \pm \infty} B(x) = B_{\pm} .
    $$
\end{itemize}

 Depending on the context, we may further assume that the magnetic field $B$ is in $ \mathcal{C}^1(\rd;\re)$.

    We will call such a magnetic field a $(B_-,B_+)$-magnetic field.
    Introduce the magnetic potential $\cA_{\mathrm{GIw}}^{(B_-,B_+)} = (A_1, A_2)$ with
    \beq\label{beta}
    A_1 = 0, \quad  A_2  = \beta(x): = \int_0^{x} B(s) ds, \quad x
    \in \re.
    \eeq
$$
H(\cA_{\mathrm{GIw}}^{(B_-,B_+)}) : = - \frac{\partial^2}{\partial x^2} + \left(-i\frac{\partial}{\partial y} - \beta(x)\right)^2.
$$

Obviously, $\cA_{\mathrm{GIw}}^{(B_-,B_+)}$ generates $B$.  When $B_-B_+>0$, i.e. in the case considered in the original work by Akira Iwatsuka \cite{I}, we will use the shorter term {\em Iwatsuka Hamiltonian}, and will write $\cA_{\mathrm{Iw}}^{(B_-,B_+)}$ or simply $\cA_{\mathrm{Iw}}$. This corresponds to a one edge situation, for the magnetic field difference $B_+-B_-$ creates a classically forbidden half space for low enough electron energy. The case $B_-B_+<0$ corresponds to a magnetic wave guide, with two classically forbidden regions at $x=\pm \infty$.

 \subsection{Quantization of edge currents for  generalized Iwatsuka Hamiltonians}

It turns out that the edge conductance can be explicitly computed for generalized Iwatsuka Hamiltonians.

\begin{theorem}\label{thmvalue}
Let $\cA_{\mathrm{GIw}}^{(B_-,B_+)}$ be a  generalized Iwatsuka potential. Let $I$ be an interval such that for some integers $n_-,n_+\ge 0$  we have
\beq\label{Igap}
I\subset \G_{n_-}(B_-)\cap \G_{n_+}(B_+).
\eeq
Then
\beq
 \sigma_e^{(I)}(H(\cA_{\mathrm{GIw}}^{(B_-,B_+)}))=(\mathrm{sign} \,B_-)n_-  -(\mathrm{sign} \, B_+)n_+.
\eeq
\end{theorem}

\begin{corollary}\label{cor1edge}
Assume $0<B_-<B_+$. Consider the Iwatsuka potential
$\cA_{\mathrm{Iw}}^{(B_-,B_+)}$. Let $I\subset \G_n^{(B_-)} \cap
(-\infty, B_+)$, for some integer $n\ge 0$.  Then \beq
 \sigma_e^{(I)}(H(\cA_{\mathrm{Iw}}^{(B_-,B_+)}))=n.
\eeq
\end{corollary}

\begin{corollary}\label{corsnake}
Assume $B_-<0<B_+$, and consider the generalized Iwatsuka
potential $\cA_{\mathrm{GIw}}^{(B_-,B_+)}$. Let $I$ be an interval
and $n_-,n_+\in\N$  such that $I\subset \G_{n_-}(B_-)\cap
\G_{n_+}(B_+)$. Then \beq
\sigma_e^{(I)}(H(\cA_{\mathrm{GIw}}^{(B_-,B_+)}))=-(n_- + n_+).
\eeq In particular if $B_-=-B_+$ and $I\subset \G_{n}(B_-)$, then
$\sigma_e(H(\cA_{\mathrm{GIw}}^{(B_-,B_+)}))=-2n$.
\end{corollary}

\begin{remark}
Corollary~\ref{corsnake} describes a purely magnetic phenomenon. With electric barriers, the net current is always zero \cite[Corollary~1]{GK}, as it is the case as well for magnetic barriers induced by magnetic fields of constant signs (see Corollary~\ref{corstrip} below). And if no magnetic field is present inside the strip, then currents are not quantized (the edge conductance is infinite). Such extended states are called `snake-orbit" states.
\\
In the particular case $|B_-|=-|B_+|$, it is interesting to note that the value of the edge conductance is exactly to twice the value coming from the Quantum Hall Effect.
\end{remark}

 \subsection{Stability of the edge conductance under magnetic perturbations}
We now state  results concerning
the stability of the edge conductance under purely magnetic perturbations.
The first one asserts that  magnetic perturbations supported on a strip in
the $y$ direction do no affect the edge conductance.

\begin{theorem}\label{thm1} Let $A \in {\mathcal C}^1(\rd; \rd)$.
Assume that $a\in\mathcal{C}^2(\R^2,\R^2)$ is a magnetic potential  compactly supported in the
$x$-direction and polynomially bounded in the $y$-direction. Let
$g$ be as in Definition~\ref{defconduc}, with $\supp g'\subset I$. Then
\begin{align} \label{kaa}
g'(H(A+a))[H(A+a),\chi] - g'(H(A))[H(A),\chi] \in \mathcal{T}_1
\end{align}
where $\mathcal{T}_1$ denotes the trace class.
Moreover, if  $g'(H(A))[H(A),\chi]\in \mathcal{T}_1$, then
\begin{equation}
  \sigma_e^{(I)}(H(A+a)) =  \sigma_e^{(I)}(H(A)).
\end{equation}
In particular, if $\cA_{\mathrm{GIw}}^{(B_-,B_+)}$ generates a $(B_-,B_+)$-magnetic field, and $I$ is as in \eqref{Igap}, then
\beq
\sigma_e^{(I)}(H(\cA_{\mathrm{GIw}}^{(B_-,B_+)}+a))= (\mathrm{sign} \, B_-)n_- -(\mathrm{sign} \, B_+)n_+ .
\eeq
\end{theorem}

\begin{remark}\label{remKAa}
In \cite[Theorem~1]{CG}, the analog of the difference \eqref{kaa}
is not only trace class but its trace automatically vanishes. This
is not the case here for the velocity operators  defined for
$H(A)$ and $H(A+a)$ differ. Indeed,
\begin{align}\label{decompremKAa}
& g'(H(A+a))[H(A+a),\chi] - g'(H(A))[H(A),\chi]  \\
& =\left(g'(H(A+a)- g'(H(A))\right)[H(A),\chi]  +
g'(H(A+a))[H(A+a)-H(A),\chi].\notag
\end{align}
The second term on the r.h.s. is due to the magnetic nature of the
perturbation, and may lead to a non trivial contribution to the
current since a direct computation yields $2ig'(H(A+a))a_2\chi'
\in \mathcal{T}_1$. To cancel this extra term, we shall introduce
a suitable gauge transform that will make $a_2$ vanish. To perform
that gauge transform, we assume a bit more than in \cite{CG},
namely, we assume that $g'(H(A))[H(A),\chi]$ (or equivalently
$g'(H(A+a)))[H(A+a),\chi]$) is  trace class.
\end{remark}

\begin{corollary}\label{corstrip}
Let $\cA_{\mathrm{strip}}\in\mathcal{C}^1(\R^2,\R^2)$ be a vector potential generating a
magnetic  field $B \in {\mathcal C}^2(\rd;\re)$ satisfying $B(x,y)= B(x)\ge B_0$ for all
$(x,y)\in\{|x|\ge R_0\}$, $R_0>0$. Then for any magnetic potential
$a \in \mathcal{C}^2(\R^2,\R^2)$ supported on $\{|x|\le R_0\}$,
and for any closed interval $I\subset]-\infty,B_0[$, we have
$\sigma_e^{(I)}(H(\cA_{\mathrm{strip}}+a))=0$.
\end{corollary}

\begin{remark}\label{julymorning}
(i) If we perturb the operator $H(\cA_{\mathrm{strip}})$ by a {\em
magnetic field} supported on a strip $S$, then Proposition
\ref{p12} below implies that there exists a potential ${\mathcal A}$ which
generates this magnetic field, and vanishes outside $S$, so that
the hypotheses of Corollary \ref{corstrip} are satisfied.\\
 (ii) The magnetic potentials  of
Corollary~\ref{corstrip} can be produced by the superposition of
two Iwatsuka-type potentials $\cA_{\mathrm{Iw}}^{(L)},\cA_{\mathrm{Iw}}^{(R)}$,
generating respectively a decreasing magnetic field $B^{(L)} =
B^{(L)}(x)$ with upper limit $B_+^{(L)} \geq B_0$ at $-\infty$, and an
increasing magnetic field $B^{(R)} = B^{(R)}(x)$ with upper limit
$B_+^{(R)} \geq B_0$ at $+\infty$. Particles are then trapped in a magnetic
strip created by these two magnetic barriers and thus can only
travel along the axis of the strip. Corollary~\ref{corstrip}
asserts that no net current can flow in such a strip, whatever the
potential inside the strip is.
The situation is very different from the  case of asymptotic values of $B$ of opposite sign, where the particle is constrained to a strip as well,
but where a net current does exist, and  we could talk about a quantum wave guide.
\end{remark}

Our second theorem provides a sum rule which is similar to one derived in
\cite{CG}. We use the convenient notation
\beq
H(A^{(1)},A^{(2)}):=H(A^{(0)}+A^{(1)}+A^{(2)}),
\eeq
where $A^{(1)}$ (resp., $A^{(2)}$), is supported in the half-plane
$x<R_1$ (resp., $x>R_2$), and $A^{(0)} : = \left(-\frac{By}{2},
\frac{B x}{2}\right)$ with  $B > 0$ corresponds to a reference Landau potential. In particular, $H(0,0)$
is the Landau Hamiltonian with constant magnetic field
$B$.

\begin{theorem}\label{thm2}
Let $I$ be a closed interval such that $I\cap\sigma(H(0,0)) =
\emptyset$, $a_\alpha, a_\beta \in {\mathcal C}^1(\rd; \rd)$. Suppose that $
a_\alpha$ (resp., $a_\beta$), is supported in the half-plane
$x<R_1$ (resp., $x>R_2$). Set
\begin{align}
\K(\al,\be):=& g'(H(\al,\be))i[H(\al,\be),\chi]-g'(H(\al,0))i[H(\al,0),\chi]\notag\\
&-g'(H(0,\be))i[H(0,\be),\chi]. \label{setK}
\end{align}
Then $\K(\al,\be)$ is trace class. Moreover, if two out of the
three terms of the r.h.s. of \eqref{setK} are trace class, then $\tr
\K(\al,\be)=0$; in particular,
\beq\label{sumab}
\sigma_e^{(I)}(H(\al,\be)) = \sigma_e^{(I)}(H(\al,0)) +
\sigma_e^{(I)}(H(0,\be)).
\eeq
Moreover, if
$\cA_{\mathrm{Iw}}^{(L)},\cA_{\mathrm{Iw}}^{(R)}$ are left and right Iwatsuka-type
potentials described in Remark \ref{julymorning},
then the same result holds for
\begin{align}
&\K'(\al,\be):= g'(H(\al,\be))i[H(\al,\be),\chi]-g'(H(\al,\cA_{\mathrm{Iw}}^{(R)}))i[H(\al,\cA_{\mathrm{Iw}}^{(R)}),\chi]\notag\\
&\;  -g'(H(\cA_{\mathrm{Iw}}^{(L)},\be))i[H(\cA_{\mathrm{Iw}}^{(L)},\be),\chi]+ g'(H(\cA_{\mathrm{Iw}}^{(L)},\cA_{\mathrm{Iw}}^{(R)}))i[H(\cA_{\mathrm{Iw}}^{(L)},\cA_{\mathrm{Iw}}^{(R)}),\chi].
\label{may1}
\end{align}
In particular, if  $I\subset
]-\infty, \inf (B_+^{(L)},B_+^{(R)})[$  and if two out of the first three  terms on the r.h.s. of \eqref{may1} are
trace class, then \beq \sigma_e^{(I)}(H(\al,\be)) =
\sigma_e^{(I)}(H(\al,\cA_{\mathrm{Iw}}^{(R)})) +
\sigma_e^{(I)}(H(\cA_{\mathrm{Iw}}^{(L)},\be)) \label{sumrule2}
\eeq
\end{theorem}

As a consequence of Corollary~\ref{cor1edge} and
Theorem~\ref{thm2}, we obtain a quantization of the edge
conductance for magnetic perturbations of the Iwatsuka
Hamiltonian, which in its turn implies the existence of edge
states flowing in the $y$ direction.

\begin{corollary}\label{corcurrent}
Assume $\cA_{\mathrm{Iw}}$ generates a $(B_-,B_+)$-magnetic field with
$B_+\ge 3B_->0$. Let $a\in \mathcal{C}^2(\{x\le R_1\} \times \R)$
for some $R_1<\infty$ be so that $\|a\|_\infty\le K_1 \sqrt{B_-}$
and $\|\mathrm{div}a\|_\infty\le K_2 B_-$, then there exists
$0<K_0<\infty$ such that
$$
 \sigma_e^{(I)} (H(a,\cA_{\mathrm{Iw}}))=n, \quad n \in {\mathbb N},
$$
for $B_-$ large enough and  an interval  $I\subset]-\infty, B_+]$
satisfying
 \beq
  I\subset ](2n-1)B_- + K_0
d_n(a,B_-),(2n+1)B_- -  K_0d_n(a,B_-)[ ,
\eeq
if $n\in\N^\ast : = \{1,2,\ldots\}$, or
 \beq
  I\subset
]-\infty,B_- -  K_0d_n(a,B_-)[,
\eeq
if $n=0$, where
$d_n(a,B_-)=\max\{\|\mathrm{div}a\|_\infty,\|a\|_\infty\sqrt{(n+1)B_-}\}$.
\end{corollary}

If the interval $I$ does not lie in a gap of the perturbed Hamiltonian
anymore, we have to introduce a regularization of the edge
conductance (see Section \ref{sectregul}).

As a  remark we note that similar results can be
obtained for Hamiltonians mixing the point of view of \cite{CG}
with purely electric potentials (wall and perturbation), and the one
of this work that is purely magnetic potentials (wall and
perturbation). We can indeed perturb an Iwatsuka Hamiltonian by an
electric potential, or perturb a Hamiltonian with an electric
confining potential by a magnetic potential. Proofs are then
similar to those of \cite{CG} and those of the present article,
the most technical case being the purely magnetic model.

\section{Spectral properties of the generalized Iwatsuka Hamiltonians}
\label{sectIw}

Denote by ${\mathcal F}$
the partial Fourier transform with respect to $y$, i.e.
    \beq\label{fourier}
    ({\mathcal F}u)(x,k) = (2\pi)^{-1/2} \int_{\re} e^{-iky}
    u(x,y) dy, \quad u \in L^2(\rd).
    \eeq
    Then  the generalized Iwatsuka Hamiltonian  is unitarily equivalent to a direct
    integral of operators with discrete spectrum, i.e.
    \beq\label{directint}
{\mathcal F}  H(\cA_{\mathrm{GIw}}^{(B_-,B_+)}) {\mathcal F}^* =
\int_{\re}^{\oplus} h(k) dk ,
    \eeq
    where
    \beq\label{hk}
    h(k) : = - \frac{d^2}{dx^2} + (k-\beta(x))^2, \quad k \in \re.
    \eeq

    For each $k \in \re$ the spectrum of $h(k)$ is discrete and
simple. Let $\left\{E_j(k)\right\}_{j \in {\mathbb N}}$ be the
increasing sequence of the eigenvalues of the operator $h(k)$, $k
\in \re$. By the Kato  perturbation theory, $E_j$ are real
analytic functions. Evidently,
$$
\sigma(H(\cA_{\mathrm{GIw}}^{(B_-,B_+)})) =
\overline{\bigcup_{j=1}^{\infty}E_j(\re)} ,
$$
where  $\sigma(H(\cA_{\mathrm{GIw}}^{(B_-,B_+)}))$ denotes the
spectrum of $H$.

    In the next proposition we summarize for further references several spectral
    properties of the Iwatsuka Hamiltonian.
    \begin{proposition} \label{propertiesIw} {\rm \cite[Lemmas 2,3, 4.1]{I}}
    Pick $0<B_-<B_+$ and let $\cA_{\mathrm{Iw}}^{(B_-,B_+)}$ be a $(B_-,B_+)$
    Iwatsuka potential, and $\{E_j(k)\}_{j = 1}^{\infty}$ be the eigenvalues defined above. Then we have
    \bel{Ebound}
    (2j-1)B_- \leq E_j(k) \leq (2j-1)B_+, \quad j \in
    {\mathbb N}^*, \quad k \in \re,
    \ee
    and
    \bel{limE}
    \lim_{k \to \pm \infty} E_j(k) = (2j - 1) B_{\pm}, \quad j \in
    {\mathbb N}^*.
    \ee
    As a consequence,
    \beq
    \sigma(H(\cA_{\mathrm{Iw}}^{(B_-,B_+)})) = \bigcup_{j=1}^{\infty} [(2j-1)B_-, (2j-1)B_+].
    \eeq
    In particular, if $B_+\ge 3B_-$, then  $\sigma(H(\cA_{\mathrm{Iw}}^{(B_-,B_+)})) = [B_-, \infty[$.
    \end{proposition}

Assume now that $B_- < 0$ and $B_+ > 0$. Let
$\left\{\mu_j\right\}_{j=1}^{\infty}$ be the non-decreasing
sequence of the eigenvalues of the operator
$$
\left(-\frac{d^2}{dx^2} + B_-^2 x^2\right) \oplus
\left(-\frac{d^2}{dx^2} + B_+^2 x^2\right),
$$
where $-\frac{d^2}{dx^2} + B_{\pm}^2 x^2$ are harmonic
oscillators, self-adjoint in $L^2(\re)$ and essentially
self-adjoint on ${\mathcal C}_0^{\infty}(\re)$.

\begin{proposition} \label{p1} Let $B_- < 0$ and $B_+ > 0$. \\
(i) For each $j \in {\mathbb N}^*$ we have
    \bel{12}
    \lim_{k \to -\infty} E_j(k) = +\infty.
    \ee
    (ii) Assume moreover $\lim_{x \to \pm \infty} B'(x) = 0$.
    Then for each $j \in {\mathbb N}^*$  we have
    \bel{13}
    \lim_{k \to \infty} E_j(k) = \mu_j.
    \ee
\end{proposition}
\begin{proof}
Relation \eqref{12} follows easily from the mini-max principle (see also the proof of \cite[Theorem 6.6]{CFKS}).\\
The argument leading to \eqref{13} goes along the general lines of
the proof of \cite[Theorem 11.1]{CFKS} (see also the proof of
\cite[Proposition 3.6]{y}).
\end{proof}

 Now we are in position to prove Theorem~\ref{thmvalue}.
Arguing as in the proof of \cite[Proposition 1]{CG}\footnote{Note that in equations (3.5) -- (3.6) of \cite{CG} there is a missing factor $-\frac{1}{2\pi}$ at the r.h.s. of (3.6).} , we have
     \begin{align}
     -2\pi \tr (g'(H(\cA_{\mathrm{GIw}}^{(B_-,B_+)})) i[H(\cA_{\mathrm{GIw}}^{(B_-,B_+)}), \chi]) &=
     -\sum_{j \in {\mathbb N}} \int_{\re}
     g'(E_j(k)) E_j'(k) dk  \\
     &=
\sum_{j \in {\mathbb N}}g(E_j(-\infty)) - g(E_j(+\infty)).
   \end{align}
The result then follows from the spectral properties of generalized Iwatsuka Hamiltonians, namely from
\eqref{limE} of Proposition~\ref{propertiesIw} if $B_-B_+>0$, and
from Proposition~\ref{p1}  if $B_-B_+<0$.


\section{Perturbation by a magnetic potential supported on a strip}
\label{sectcompact}

  \subsection{More on magnetic fields and magnetic potentials}
This subsection contains  well-known facts about the possibility
 to construct magnetic potentials $A:
\rd \to \rd$ with prescribed properties, which generate given
magnetic fields $B: \re^2 \to \re$.
    In the first proposition we define a magnetic potential $A$ in
the so-called {\em Poincar\'e gauge}.
\begin{proposition} \label{p11} {\rm \cite[Eq. (8.154)]{T}}
 Let  $B \in {\mathcal C}^1(\rd;\re)$. Then the potential
    \bel{g70}
A = (A_1, A_2) = \left(-y\int_0^1 sB(sx,sy) ds, \; x\int_0^1
sB(sx,sy) ds\right), \quad (x,y) \in \rd,
    \ee
generates the magnetic field $B$.
\end{proposition}

\begin{proposition} \label{p12}
Let $B \in {\mathcal C}^k(\rd;\re)$, $k \in {\mathbb N}^*$, satisfy $B({\bf x}) = 0$ for
${\bf x} = (x,y) \in \re^2$ with $|x| \geq R_0$, $R_0 > 0$. Then
there exists a magnetic potential ${\mathcal A} \in {\mathcal C}^k(\rd;\rd)$ which generates
$B$, and vanishes identically on $\{(x,y) \in \rd \, | \, |x| \geq
R_0\}$.
\end{proposition}
\begin{proof}
Pick any magnetic potential $A \in {\mathcal C}^k(\rd;\rd)$ which
generates $B$ (say, the potential appearing in \eqref{g70}). Set
$$
S_- : = \{(x,y) \in \rd \, | \, x <  -R_0\}, \quad S_+ : = \{(x,y)
\in \rd \, | \, x >  R_0\}.
$$
Since $S_{\pm}$ are simply connected domains, and $B$ identically
vanishes on them, there exist functions $F_{\pm} \in {\mathcal
C}^{k+1}(\overline{S}_{\pm}; \re)$ such that
$$
\nabla F_- = A \quad {\rm on} \quad S_-, \quad \nabla F_+ = A
\quad {\rm on} \quad S_+.
$$
Then there exists an extension ${\mathcal F} \in {\mathcal
C}^{k+1}(\rd; \re)$ such that
$$
{\mathcal F} = F_-  \quad {\rm on} \quad S_-, \quad {\mathcal F} =
F_+  \quad {\rm on} \quad S_+.
$$
On $\rd$ define ${\mathcal A} : = A - \nabla {\mathcal F}$.
Evidently,
  the magnetic potential ${\mathcal A} \in {\mathcal C}^k(\rd;\rd)$
  generates $B$, and ${\mathcal A}({\bf x}) = 0$ for ${\bf x} \in S_- \cup S_+$.
\end{proof}

\subsection{Proof of Theorem~\ref{thm1} and Corollary~\ref{corstrip} }

\begin{lemma}\label{lemgauge}
Let $A \in {\mathcal C}^2(\rd;\rd)$. Assume that $a \in {\mathcal
C}^1(\rd; \rd)$  is supported on the strip $[-x_0,x_0]\times \R$
and admits the bound $|a(x,y)|\le C_a \scal{y}^k$, for some $k\ge
0$ and $C_a<\infty$. Set
    \bel{g46}
    F(x,y) = -\int_0^y a_2(x,s) ds, \quad (x,y) \in \rd.
\ee
    Then we have
    \bel{g47}
    [H(A+a+\nabla F),\chi]=[H(A),\chi].
\end{equation}
\end{lemma}
\begin{proof}
Note that if $\tilde{a}=(\tilde{a}_1,\tilde{a}_2) \in
\mathcal{C}^1(\rd, \rd)$, we have
 \beq \label{expand}
    H(A+\tilde{a}) - H(A) = 2 \tilde{a}\cdot  (i\nabla + A) + i {\rm div}\,\tilde{a} + |\tilde{a}|^2.
\eeq Since $\partial_x \chi=0$, a direct computation shows that
    \bel{g45}
    [H(A+\tilde{a}),\chi] - [H(A),\chi] = 2i \tilde{a}\cdot \nabla
    \chi = 2i \tilde{a}_2\chi'.
    \ee
    Therefore,
\beq\label{condcom} [H(A+\tilde{a}),\chi]-[H(A),\chi]=0
\Longleftrightarrow \tilde{a}_2\chi'=0. \eeq Applying \eqref{g45}
-- \eqref{condcom} with $\tilde{a} = a + \nabla F$, and taking
into account that in this case $\tilde{a}_2 = a_2 + \partial_y F =
0$ by \eqref{g46}, we obtain \eqref{g47}.
\end{proof}

\begin{remark}\label{remstrip}
Note that $a+\nabla F$ is supported on $[-x_0,x_0]\times \R$, and
$|a+\nabla F| \leq C_a \scal{y}^{k+1}$, $(x,y) \in \rd$.
\end{remark}

\begin{proposition}\label{propstrip} Let  $A, a$ be as in Lemma \ref{lemgauge}.
Then
$$
\left(g'(H(A+a ))-
g'(H(A))\right)[H(A),\chi] \in \mathcal{T}_1.
$$
Moreover if $[H(A+a),\chi]=[H(A),\chi]$, then
\begin{equation}
\tr\left(g'(H(A+a))- g'(H(A))\right)[H(A),\chi] =0. \label{may2}
\end{equation}
\end{proposition}

Assuming for the moment the validity of
Proposition~\ref{propstrip}, we provide the proof of
Theorem~\ref{thm1}.

\begin{proof}[Proof of Theorem~\ref{thm1}]
The fact that the operator defined in \eqref{kaa} is trace class
follows from the decomposition \eqref{decompremKAa} of
Remark~\ref{remKAa}, the fact that $2ig'(H(A+a)) a_2 \chi' \in {\mathcal T}_1$, and the first part of Proposition~\ref{propstrip}.\\
 Further, in order to prove \eqref{may2} we introduce  a
gauge transform $\exp(iF)$, where $F$ is given by
Lemma~\ref{lemgauge}. By Proposition~\ref{propstrip} applied to
the perturbation $a+\nabla F$,  the operator
$$
g'(H(A+a+\nabla
F))[H(A),\chi]=g'(H(A+a+\nabla F))[H(A+a+\nabla F),\chi]
$$
is trace-class since $g'(H(A))[H(A),\chi]$ is by the hypotheses of Theorem \ref{thm1}.

Now, since $F$ and $\chi$ commute, we have
\begin{align}
& g'(H(A+a))i[H(A+a), \chi] \\
&=
e^{-iF}(e^{iF}g'(H(A+a))e^{-iF} e^{iF} i[H(A+a), \chi]e^{-iF}) e^{iF}\\
& =  e^{-iF}(g'(H(A+a + \nabla F))i[H(A+a + \nabla F), \chi])e^{iF},
\end{align}
which is trace class, so that $g'(H(A+a))i[H(A+a), \chi]\in\mathcal{T}_1$. It follows,
using the cyclicity of the trace, that
\begin{align}
&\tr g'(H(A+a))i[H(A+a), \chi] - \tr g'(H(A))i[H(A), \chi] \\
& \quad=
\tr (g'(H(A+a + \nabla F)) - g'(H(A)))i[H(A), \chi] \\
&\quad =0,
\end{align}
by Lemma~\ref{lemgauge} and Proposition~\ref{propstrip}.
\end{proof}

The rest of the section is devoted to the proofs of
Proposition~\ref{propstrip} and Corollary~\ref{corstrip}.

\begin{proof}[Proof of Proposition~\ref{propstrip}]
 Let $\varphi_r:\R^2\to\R^2$ be the smooth cut-off function
 satisfying $\varphi_r(x,y)=\varphi_r(y)$,
 $\varphi_r = 1$ for $|y|\le r-1$ and $\varphi_r = 0$ for $|y|\ge r$, $r>1$.
 We decompose ${a}= {a} \varphi_r +
{a}(1-\varphi_r) := a_{\le r} + a_{\ge r}$.
 Then we have
$$
\tr \left(g'(H(A+{a}))- g'(H(A))\right) [H(A),\chi] =
$$
    $$
    \tr\left(g'(H(A+{a}))- g'(H(A+a_{\le r}))\right)
[H(A),\chi]  +
$$
    \bel{g50}
 \tr\left(g'(H(A+a_{\le r}))- g'(H(A))\right) [H(A),\chi].
    \ee
    First, we will show that
    \bel{g51}
    g'(H(A+a_{\le r}))- g'(H(A) [H(A),\chi] \in
    \mathcal{T}_1,
    \ee
    and
    \bel{g52}
 \tr\left(g'(H(A+a_{\le r}))- g'(H(A))\right) [H(A),\chi]  = 0.
    \ee
      After that we will show that for some $C<\infty$ and $p\ge 1$,
    \beq\label{decaystrip}
    \left\|\left(g'(H(A+{a}))- g'(H(A+a_{\le
r}))\right) [H(A),\chi] \right\|_1 \le C r^{-p}
    \eeq
    for  $r$ large enough.\\
Let us now prove \eqref{g51}.
 By the
Helffer-Sj\"ostrand functional calculus (see e.g.
\cite[Lemma~B.2]{HS}), applied to function
$G(x):=-\int_{x}^{\infty} g(s)ds$, we have
\begin{eqnarray}\label{HS3}
g'(H(A+a_{\le r}))- g'(H(A))=
-\dfrac{2}{\pi}\int_{\rd} \overline{\partial}
\tilde{G}(u+iv)(R_1^3-R_2^3)dudv,
\end{eqnarray}
where $R_1=(H_j-z)^{-1}$, $j=1,2$, and $H_1 : = H(A)$, $H_2 : =
H(A+a_{\le r})$. Put
$$
\W:=H_2-H_1= 2a_{\le r}\cdot(i\nabla +A) + i {\rm
div}\,a_{\le r} + |a_{\le r}|^2.
$$
Note that due to the fact that  $\W$ is a first-order differential operator,
we need one extra power of the resolvents in comparison to
\cite{CG}. Thus, we have to analyze the operator $(R_1^3-R_2^3)[H(A),\chi]$. It is easy to check that
\begin{align}
          2( R_1^3-R_2^3)=
          &R_1^2\W R_2R_1+R_1\W R_2^2R_1+R_1\W R_2R_1^2\\
          &+R_2^2 R_1\W R_2+R_2 R_1^2\W R_2+R_2 R_1\W R_2^2,
\end{align}
which could be formally  guessed by  computing the derivative
$\partial_{z}( R_1^2-R_2^2)=\partial_{z}(R_1\W R_2R_1+R_2 R_1\W
R_2)$. We have
\begin{align}
&2(R_1^3-R_2^3)[H_1,\chi]  \\
& = \left( R_1^2\W R_2R_1+R_1\W R_2^2R_1+R_1\W R_2R_1^2\right) [H_1,\chi]  \label{cube1}
  \\
& \quad+  \left(R_2^2 R_1\W R_2+R_2 R_1^2\W R_2+R_2 R_1\W R_2^2\right) [H_2,\chi] \label{cube2} \\
 &\quad -  \left(R_2^2 R_1\W R_2+R_2 R_1^2\W R_2+R_2 R_1\W R_2^2\right) [\W,\chi] \label{cube3}.
 \end{align}
We recall that $\,a_{\le r}$ is compactly supported. Applying Corollary \ref{gf1} and Lemma \ref{gl1} (ii), and bearing
in mind that the operators $R_j[H_j,\chi]$, $j=1,2$, are bounded,
we find that all the terms on the r.h.s. of \eqref{cube1}- \eqref{cube3}  are
trace-class, which combined with \eqref{HS3} implies \eqref{g51}.

Next, we prove \eqref{g52}. Using the identities $R_j[H_j,\chi]R_j = [\chi,R_j]$, $j=1,2$, undoing the commutators, and introducing obvious notations, we get
        \bel{g62}
              {\rm tr}\,(R_1^2\W R_2R_1)[H_1,\chi] = {\rm tr}\, [\chi,R_1] R_1 W R_2 =: I_1+I_2,
              \ee
               \bel{g63}
              {\rm tr}\,(R_1\W R_2^2R_1)[H_1,\chi] = {\rm tr}\, [\chi,R_1] W R_2^2 =: II_1+II_2,
              \ee
               \bel{g64}
              {\rm tr}\,(R_1\W R_2R_1^2)[H_1,\chi] = {\rm tr}\, [\chi,R_1]  W R_2 R_1 =: III_1+III_2,
              \ee
        \bel{g65}
              {\rm tr}\,(R_2^2 R_1\W R_2)[H_2,\chi]= {\rm tr}\, [\chi,R_2] R_2 R_1 W =: IV_1+IV_2,
              \ee
               \bel{g66}
              {\rm tr}\,(R_2 R_1^2\W R_2)[H_2,\chi]= {\rm tr}\, [\chi,R_2] R_1^2 W  =: V_1+V_2,
              \ee
               \bel{g67}
              {\rm tr}\,(R_2 R_1\W R_2^2)[H_2,\chi]= {\rm tr}\, [\chi,R_2] R_1 W R_2 =: VI_1+VI_2.
              \ee
            Let $0 \leq \zeta_j \in {\mathcal C}_0^{\infty}(\rd)$, $j = 0,1$, satisfy $\zeta_0 a_{\leq r} = a_{\leq r}$ and $\zeta_1 \zeta_0 = \zeta_0$ on $\rd$. Rearranging the terms \eqref{g62} -- \eqref{g67} of $2 \, {\rm tr}\,(R_1^2 - R_2^3) [H_1, \chi]$  and applying Lemma \ref{gl1} and Lemma \ref{gl3}, we get
            \begin{align}
            I_1 + V_2 & =
            {\rm tr} [\chi R_1^2 W, R_2] = {\rm tr} [\chi R_1^2 W, \zeta_0 R_2] = 0, \\
            II_1 + V_2  &=
            {\rm tr}[\chi R_1 W R_2, R_2] \\
            &= {\rm tr}\left( [\chi R_1 W R_2, \zeta_0 R_2] - [\chi R_1 W R_2 [H_2, \zeta_0] R_2, R_2]\right) = 0,
            \\
            III_1 + I_2 & =
            {\rm tr}[\chi R_1 W R_2, R_1] \\
            & = {\rm tr}\left( [\chi R_1 W R_2, \zeta_0 R_1] - [\chi R_1 W R_2 [H_2, \zeta_0] R_2, R_1]\right) = 0,
            \\
            IV_1 + II_2 &=
            {\rm tr}[\chi R_2^2, R_1 W \chi] \\
        &= {\rm tr}\left( [\chi R_2 \zeta_0 R_2, R_1 W \chi] - [\zeta_0 R_2 [H_2,\zeta_1] \chi R_2^2, R_1 W \chi]\right) = 0,
            \\
            V_1 + III_2 &=
            {\rm tr}[\chi R_2 R_1, R_1 W \chi] \\ &= {\rm tr}\left( [\chi R_2 \zeta_0 R_1, R_1 W \chi] - [\zeta_0 R_2 [H_2,\zeta_1] \chi R_2 R_1, R_1 W \chi]\right) = 0,
            \\
            VI_1 + IV_2 &=
            {\rm tr} [\chi R_2 R_1 W, R_2] = {\rm tr} [\chi R_2 R_1 W, \zeta_0 R_2] = 0.
       \end{align}
Therefore,
    \begin{equation}
    \eqref{cube1} +\eqref{cube2} = 0.
    \ee
 To get \eqref{g52} it remains to see that $\eqref{cube3}=0$, but this is immediate because by assumption $[H(A+a)-H(A),\chi]=0$, which readily implies that
 \begin{align}
  [\W,\chi] & =-[H(A+a)-H(A+ a_{\le r}),\chi]  \\
  &= - [2a_{\ge r}\cdot(i\nabla +A) , \chi]  = 0,
    \end{align}
since supports of $\chi'$ and $a_{\ge r}$ are disjoint.

 Finally, we prove \eqref{decaystrip}. Due to the
Helffer-Sj\"ostrand formula \eqref{HS3}, we have to control
$$
\|(R_1^3-R_2^3)[H(A),\chi]\|_1,
$$
 where we use the notations $R(z):=(H(A)-z)^{-1}$ and $R_r
:= (H(A+a_{\le r})-z)^{-1}$. The resolvent identity yields
    \beq
R^3 - R_r^3 = R^3 \W_r R_r + R^2\W_r R_r^2 +  R\W_r
R_r^3,
    \label{idresolv3}
    \eeq
where
$$
\W_r := H(A+a_{\le r})-H(A+{a})
 = -\left(2a_{\ge r} \cdot (i\nabla + A) + i\div (a_{\ge r}) + |{a}|^2 - |a_{\le r}|^2\right).
$$
Note that $\W_r\equiv 0$ whenever $|y|\le r-1$, so that we write
\begin{align}
\W_r& =\sum_{(x_1,y_1)\in\Z^2\cap[-x_0-1,x_0+1]\times [-r+2,r-2]^c} \1_{(x_1,y_1)} \W_r \\
& = \sum_{(x_1,y_1)\in\Z^2\cap[-x_0-1,x_0+1]\times [-r+2,r-2]^c}  \W_r \1_{(x_1,y_1)},
\label{decompW}
\end{align}
where $\1_{(x,y)}$ stands for a smooth characteristic function of
the cube of side length one and centered at $(x,y) \in \rd$ such that
$\sum_{(x,y)\in\Z^2}\1_{(x,y)} = 1$. Similarly,
    \beq
[H(A),\chi]=\sum_{x_2\in\Z} [H(A),\chi] \1_{(x_2,0)}.
\label{decompcom}
    \eeq
    For the moment fix $x_1, y_1, x_2 \in {\mathbb Z}$, and
    introduce the short-hand notations $\zeta_0 : =
    \1_{(x_1,y_1)}$, and $\zeta : =
    \1_{(x_2,0)}$. Let $\zeta_j$ be non-negative
    smooth compactly supported functions such that $\zeta_j
    \zeta_{j-1} = \zeta_{j-1}$ on $\rd$, $j=1,2,3$. Then we have
    \bel{g55}
\Vert R^3\1_{(x_1,y_1)}\W_r R_r [H(A),\chi]\1_{x_2}\Vert_1
     \leq
      \Vert R^3
\zeta_0 \W_r
    \Vert_1 \|\zeta_1 R_r[H(A),\chi]\zeta\Vert .
    \ee
Next,
   \begin{align}
&\Vert R^2 \1_{(x_1,y_1)}\W_r R_r^2 [H(A),\chi]\1_{x_2}\Vert_1  \\
& \leq
    \Vert R^2 \zeta_0 \W_r
     R_r \Vert_1 \left(\Vert \zeta_1 R_r[H(A),\chi]\zeta\Vert
 + \Vert \zeta_2 [H_r(A),\zeta_1]
     R_r^2[H(A),\chi]\zeta\Vert\right), \label{g56}
    \end{align}
and \begin{align}
& \Vert R \1_{(x_1,y_1)}\W_r R_r^3 [H(A),\chi]\1_{x_2}\Vert_1\\
 &   \leq
     \Vert R \zeta_0 \W_r
     R_r^2 \Vert_1 \left(\Vert \zeta_1 R_r[H(A),\chi]\zeta\Vert + \Vert \zeta_2 [H(A),\zeta_1]
     R_r^2[H(A),\chi]\zeta\Vert\right)\\
     & \; +
     \Vert R \zeta_0 \W_r
     R_r [H(A),\zeta_1] R_r \Vert_1 \left(\Vert \zeta_2
     R_r^2[H_r(A),\chi]\zeta\Vert + \Vert \zeta_3 [H_r(A),\zeta_2]
     R_r^3 [H(A),\chi]\zeta\Vert\right). \label{g57}
    \end{align}
    Assume now that $z$ is in a compact subset of ${\mathbb C}$,
    and $\Im{z} \neq 0$. Applying  Proposition
    \ref{gp1} and estimate \eqref{g41} below, we find that there exists a constant $c_1$ independent of
    $x_1, y_1, x_2 \in {\mathbb Z}$, and $z$, such that the
    trace-class norms $$
    \Vert R^3
\zeta_0 \W_r
    \Vert_1, \quad \Vert R^2
\zeta_0 \W_r
     R_r \Vert_1, \quad \Vert R \zeta_0 \W_r
     R_r [H(A),\zeta_1] R_r \Vert_1,
     $$
     appearing on the r.h.s. of
     \eqref{g55} -- \eqref{g57} are upper bounded by $c_1
     |\Im{z}|^{-3}$. On the other hand, making use of estimates of
     Combes-Thomas type (see \cite{CT, GK}), we find that there exists a constant $c_2
     > 0$ independent of $x_1, y_1, x_2 \in {\mathbb Z}$, and $z$ such that
     the operator norms $$
     \Vert \zeta_1
     R_r[H(A),\chi]\zeta\Vert, \quad \Vert \zeta_2 [H_r(A),\zeta_1]
     R_r^2[H(A),\chi]\zeta\Vert,
     $$
     $$
      \Vert \zeta_2
     R_r^2[H(A),\chi]\zeta\Vert,  \quad
     \Vert \zeta_3 [H_r(A),\zeta_2]
     R_r^3 [H(A),\chi]\zeta\Vert,
     $$
     appearing on the r.h.s. of
     \eqref{g55} -- \eqref{g57} are upper bounded by $$c_2
     |\Im{z}|^{-1} \exp{(-c_2 |\Im{z}|(|x_1 - x_2| + |y_1|))}.$$ Taking into account
     these estimates, bearing into mind the representations
     \eqref{idresolv3}, \eqref{decompcom}, and \eqref{decompW}, and arguing as in the
     proof of \cite[Lemma 2]{CG}, we easily obtain \eqref{decaystrip}.
    \end{proof}

\begin{proof}[Proof of Corollary~\ref{corstrip}]
We introduce a modified strip confining potential
$\tilde{\cA}_{\mathrm{strip}}$ generating a magnetic field
$\tilde{B} \in {\mathcal C}^2(\rd;\re)$ which satisfies $\tilde{B}(x,y)\ge B_0$
for all $(x,y)\in\R^2$ and $\tilde{B}(x)=B(x)$ on $\{|x|\ge
R_0\}$. Since the operator $H(\tilde{\cA}_{\mathrm{strip}}) -
\tilde{B}$ is non-negative (see
e.g. \cite{E}), we have
$\inf\,\sigma(H(\tilde{\cA}_{\mathrm{strip}})) \geq B_0$  . As a consequence,
$\sigma_e^{(I)}(H(\tilde{\cA}_{\mathrm{strip}}))=0$. Since the
magnetic field $B-\tilde{B}$ is supported on a strip,
Proposition~\ref{p12} implies the existence a magnetic potential $A \in
{\mathcal C}^2(\rd; \rd)$ which generates the magnetic field $B-\tilde{B}$,
and is supported on the strip $\{|x|\le
R_0\}$. Applying Theorem~\ref{thm1}, we find
that
$$
\sigma_e^{(I)}(H(\tilde{\cA}_{\mathrm{strip}}+A)) =
\sigma_e^{(I)}(H(\tilde{\cA}_{\mathrm{strip}})) = 0.
$$
Since the potentials $\tilde{\cA}_{\mathrm{strip}}+A$ and
${\cA}_{\mathrm{strip}}$ generate the same magnetic field $B$, the
operators $H(\tilde{\cA}_{\mathrm{strip}}+A)$ and
$H(\cA_{\mathrm{strip}})$ are unitarily equivalent under an
appropriate gauge transform. Therefore,
$$
\sigma_e^{(I)}(H(\cA_{\mathrm{strip}})) =
\sigma_e^{(I)}(H(\tilde{\cA}_{\mathrm{strip}}+A)) = 0.
$$
Finally, applying Theorem~\ref{thm1} once more, we find that
$$
\sigma_e^{(I)}(H(\cA_{\mathrm{strip}}+a)) =
\sigma_e^{(I)}(H(\cA_{\mathrm{strip}})) = 0.
$$
\end{proof}


\section{Sum rule for magnetic perturbations}
\label{sectSumrule}
The aim of this section is to prove Theorem~\ref{thm2} and Corollary~\ref{corcurrent}.

\begin{proof}[Proof of Theorem~\ref{thm2}]
It is enough to prove the first part of the statement (the one concerning $\K(\al,\be)$), for  the second part  will follow from the relation
\beq
\K'(\al,\be)= \K(\al,\be)-\K(\al,\cA_{\mathrm{Iw}}^{(L)})-\K(\cA_{\mathrm{Iw}}^{(R)},\be)+\K(\cA_{\mathrm{Iw}}^{(L)},\cA_{\mathrm{Iw}}^{(R)}),
\eeq
where we used that $\tr g'(H(\cA_{\mathrm{Iw}}^{(L)},\cA_{\mathrm{Iw}}^{(R)})) [H(\cA_{\mathrm{Iw}}^{(L)},\cA_{\mathrm{Iw}}^{(R)}),\chi]=0$ by Corollary~\ref{corstrip}.

We set $\K:=\K({\al},{\be})$ and $\K_r:=\K({\al},{\be}_r)$, where ${\be}_r:= \be \varphi_r$ and $\varphi_r$ is a smooth characteristic function of the region $x\le r$. By Theorem~\ref{thm1}, we have $\K_r\in\T_1$ and $\tr \K_r =0$. It is thus enough to prove polynomial decay in $r$ of $\|\K-\K_r\|_1$.
We set
\beq
\Q^{\chi}_{(a,b)}:=i[H(a,b),\chi].
\eeq
Note that although $H(a,b)$ is non linear in $a,b$, the commutator$\Q^{\chi}_{(a,b)}$ is linear, that is for arbitrary $a,b,a',b'$  in $\mathcal{C}^1$,
\bel{Qlin}
 \Q^{\chi}_{(a,b)}-\Q^{\chi}_{(a',b')}=\Q^{\chi}_{(a-a',b-b')}.
 \ee
We get
\begin{align}
&\K-\K_r
=
g'(H(\al,\be))\Q_{(\al,\be)}^{\chi} -g'(H(\al,{\be}_r))\Q_{(\al,{\be}_r)}^{\chi}\\
& \quad - g'(H(0,\be))\Q_{(0,{\be})}^{\chi} +g'(H(0,{\be}_r))\Q_{(0,{\be}_r)}^{\chi}
\\
&=
(g'(H(\al,\be))-g'(H(\al,{\be}_r)))\Q_{(\al,{\be}_r)}^{\chi}  \\
& \quad -(g'(H(0,\be))-g'(H(0,{\be}_r)))\Q_{(0,{\be}_r)}^{\chi} \\
&\quad +( g'(H(\al,\be)) - g'(H(0,\be)))\Q_{(0,\be-{\be}_r)}^{\chi}\\
&=
(g'(H(\al,\be))-g'(H(\al,{\be}_r))-g'(H(0,\be))+g'(H(0,{\be}_r)))\Q_{(\al,{\be}_r)}^{\chi}  \label{4terms}\\
& \quad -(g'(H(0,\be))-g'(H(0,{\be}_r)))\Q_{(\al,0)}^{\chi} \label{Q1}\\
&\quad +( g'(H(\al,\be)) - g'(H(0,\be) ))\Q_{(0,\be-{\be}_r)}^{\chi},   \label{Q2}
\end{align}
where we used \eqref{Qlin}. The term in \eqref{4terms} is evaluated as in \cite{CG}, while \eqref{Q1} and \eqref{Q2} are new terms coming from the magnetic nature of the perturbation.

We first show that \eqref{Q1} and \eqref{Q2} satisfy a bound of
the type \eqref{decaystrip}. Notice, from the very definition of
commutators $\Q^{\chi}_{(a,b)}$, that if $a,b$ given are supported
on  the closed region  $\Gamma$, then so is the operator
$\Q^{\chi}_{(a,b)}$, in the sense that
$(1-\chi_\Gamma)\Q^{\chi}_{(a,b)}=(\Q^{\chi}_{(a,b)})_{|_{\R^2\setminus
\Gamma}} \equiv 0$,  $\chi_\Gamma$ being the characteristic
function of $\Gamma$.

Let us first consider \eqref{Q1}. We use the Helffer-Sj\"ostrand formula \eqref{HS3} and then proceed as in the proof of Proposition~\ref{propstrip}, equations \eqref{idresolv3} and below, but  this time with
\begin{align}
\W_r
&:= H(0,{\be}_r)-H(0,\be)\\
&=
\sum_{(x_1,y_1)\in\Z^2\cap\{x_1\ge r-1\}} \1_{(x_1,y_1)} \W_r \\
& = \sum_{(x_1,y_1)\in\Z^2\cap\{x_1\ge r-1\}} \W_r \1_{(x_1,y_1)},
\end{align} and
\beq
\Q_{(\al,0)}^{\chi}=\sum_{x_2\in\Z\cap\{x_2\le 1\}}  \Q_{(\al,0)}^{\chi} \1_{(x_2,0)},
\eeq
instead of \eqref{decompW} and \eqref{decompcom}. The proof of \eqref{Q2} is similar but now we consider $\W:=H(\al,\be) - H(0,\be)$ which we decompose over boxes centered at points $(x_1,y_1)\in \Z^2\cap\{x_1\le 1\}$, while we decompose $\Q_{(0,\be-{\be}_r)}^{\chi}$ over boxes centered at points $(x_2,y_2)\in \Z^2\cap\{x_2\ge r-1, y_2=0\}$.

We turn now to  \eqref{4terms}. Again, in order to estimate
\begin{equation}
\left\|\left\{g'(H(\al,\be))-g'(H(\al,{\be}_r))-
g'(H(0,\be))+g'(H(0,{\be}_r))\right\}\Q_{(\al,{\be}_r)}^{\chi} \right\|_1,
\end{equation}
we bound, with obvious notations for the resolvents, the norm
\begin{equation}
\left\|\left\{\left(R^3_{(\al,\be)}- R^3_{(\al,{\be}_r)}\right) - \left( R^3_{(0,\be)}-R^3_{(0,{\be}_r)}\right)\right\}\Q_{(\al,{\be}_r)}^\chi\right\|_1.
\end{equation}
To do so, we utilize  the resolvent identity \eqref{idresolv3}
together with
\begin{equation}
 H(\al,\be)-H(\al,{\be}_r)=H(0,\be)-H(0,{\be}_r)=:-\W_r,
\end{equation}
and get
\begin{align}
&\left(R^3_{(\al,\be)}- R^3_{(\al,{\be}_r)}\right) - \left( R^3_{(0,\be)}-R^3_{(0,{\be}_r)}\right)\\
& =
\left(R^3_{(\al,\be)} \W_r R_{(\al,{\be}_r)} + R^2_{(\al,\be)}\W_r R_{(\al,{\be}_r)}^2 +  R_{(\al,\be)}\W_r R_{(\al,{\be}_r)}^3\right)\\
& \quad - \left( R_{(0,\be)}^3 \W_r R_{(0,{\be}_r)} + R_{(0,\be)}^2\W_r R_{(0,{\be}_r)}^2 +  R_{(0,\be)}\W_r R_{(0,{\be}_r)}^3 \right)\\
& =
\left(R^3_{(\al,\be)} \W_r R_{(\al,{\be}_r)} -R_{(0,\be)}^3 \W_r R_{(0,{\be}_r)}\right) \label{diff1}\\
&\quad + \left(  R^2_{(\al,\be)}\W_r R_{(\al,{\be}_r)}^2-  R_{(0,\be)}^2\W_r R_{(0,{\be}_r)}^2 \right) \label{diff2}\\
& \quad + \left( R_{(\al,\be)}\W_r R_{(\al,{\be}_r)}^3-  R_{(0,\be)}\W_r R_{(0,{\be}_r)}^3 \right). \label{diff3}
\end{align}
Next, with $\W_{\be}=H(0,\be)-H(\al,\be)$ and $\W_{{\be}_r}= H(0,{\be}_r)-H(\al,{\be}_r)$, we rewrite \eqref{diff1} using the resolvent identity:
\begin{align}
& R^3_{(\al,\be)} \W_r R_{(\al,{\be}_r)} -R_{(0,\be)}^3 \W_r R_{(0,{\be}_r)} \\
&=R^3_{(\al,\be)} \W_r R_{(\al,{\be}_r)}\W_{{\be}_r}  R_{(0,{\be}_r)}) +(R^3_{(\al,\be)}-R_{(0,\be)}^3)  \W_r R_{(0,{\be}_r)}\\
&=R^3_{(\al,\be)} \W_r R_{(\al,{\be}_r)}\W_{{\be}_r}  R_{(0,{\be}_r)} +R^3_{(\al,\be)}\W_{\be} R_{(0,\be)} \W_r R_{(0,{\be}_r)}\\
&+R^2_{(\al,\be)}\W_{\be} R^2_{(0,\be)} \W_r R_{(0,{\be}_r)}+R_{(\al,\be)}\W_{\be} R^3_{(0,\be)} \W_r R_{(0,{\be}_r)}.
\end{align}
As previously, $\W_r$ is decomposed over boxes centered at
$(x_2,y_2)\in\Z^2$ such that $x_2\ge r-1$, while both $\W_{\be}$
and $\W_{{\be}_r}$ are decomposed over integers $(x_3,y_3)$'s such
that $x_3\le 1$. Proceeding as above yields the result. We then
apply the same argument as for \eqref{diff2} and \eqref{diff3}.
\end{proof}

Corollary~\ref{corcurrent} is a direct consequence of Theorem~\ref{thm2} and Lemma~\ref{lemgap} below, for it is enough to prove that $I\cap \sigma(H(a,0))=\emptyset$, which readily implies that $\sigma_e^{(I)}(H(\cA_{\mathrm{Iw}},\be))=0$.

\begin{lemma}\label{lemgap}
Let $H(A^{(0)})$ be the Landau Hamiltonian with constant magnetic filed $B_-$. Let $a\in \mathcal{C}^1(\R^2)$ be such that $\|a\|_\infty\le K_1 \sqrt{B_-}$ and $\|\mathrm{div}a\|_\infty \le K_2 B_-$. Then there exists a constant $0<K_0<\infty$ such that we have
$$
I\cap \sigma(H(A_0+a))=\emptyset
$$
for $B_-$ large enough and for any interval $I$ satisfying
\beq
I\subset ](2n-1)B_- + K_0 d_n(a,B_-),(2n+1)B_- -  K_0d_n(a,B_-)[,
\eeq
if $n\in\N^\ast$, or
\beq
I\subset ]-\infty,B_- -  K_0 d_0(a,B_-)[
\eeq
if $n=0$, where $d_n(a,B_-)=\max(\|\mathrm{div}a\|_\infty,\|a\|_\infty\sqrt{(n+1)B_-})$, $n \in {\mathbb N}$.
\end{lemma}

\begin{proof}
We denote by $R_a$ (resp. $R_0$), the resolvent of $H(A_0+a)$ (resp. $H(A_0)$), and set $\W_a=H(A_0+a)-H(A_0)$,
 We start with by resolvent identity
\begin{equation}
 R_0(E)=R_a(E)( \mathrm{id}-\W_a R_0(E))
\end{equation}
with $E\in I\subset\R\setminus\sigma(H(A_0))$.
To show that $E$ belongs to the resolvent set of $H(A^{(0)}+a)$, it is enough to show that $\|\W_a R_0(E)\|<1$ so that $\mathrm{id}-\W_a R_0(E)$ is invertible. We have
\begin{align}
\Vert\W_a R_0(E)\Vert
& \leq\Vert (\mathrm{div} a+\vert a\vert^2)R_0(E)\Vert+\Vert a \cdot (i\nabla + A^{(0)}) R_0(E)\Vert \\
& \le
(\|\mathrm{div} a\|_\infty+\|a\|^2_\infty) \|R_0(E)\| + \| a\|_\infty \|(i\nabla + A^{(0)}) R_0(E)\|.
\end{align}
We set $d=\mathrm{dist}(E,\sigma(H(0,0)))$. We have, noting that $d\le \max(|E|,B_-)$,
\begin{equation}
\|(i\nabla + A^{(0)}) R_0(E)\|^2 = \|R_0(E) H(A_0) R_0(E)\|
\le  \frac1d + \frac {\max(|E|,0)}{d^2} \le \frac{2\max(|E|,B_-)}{d^2},
\end{equation}
so that
\begin{align}
\Vert\W_a  R_{0}(E)\Vert
\le (\|\mathrm{div} a\|_\infty+\|a\|^2_\infty) \frac1d + \| a\|_\infty \frac{\sqrt{2\max(|E|,B_-)}}{d}.
\end{align}
Assuming that $E$ belongs to the $n$th band,  $n\in\N$,  and $d$
satisfies both  $d> 8 \|a\|_\infty\sqrt{(n+1)B_-}$ and $d>
\frac12(\|\mathrm{div} a\|_\infty+\|a\|^2_\infty)$, we find that
$\Vert\W_a R_{0}(E)\Vert<1$.
\end{proof}

As a final remark, we take advantage of the second sum rule \eqref{sumrule2} to sketch an alternative proof of the existence of quantized currents in magnetic guides created by magnetic barriers of opposite signs.

We first specify a given reference decreasing switch  function $\theta$ such that $\theta(x)=1$ (resp. $\theta(x)=-1$),  for $x< q_-$ (resp. $x>q_+$), with some real numbers $q_-\le q_+$. Set ${\mathcal B}(x)=B\theta(\sqrt{B}x)$ and denote by $\tilde{\beta}(x)=B\int_0^x \theta(\sqrt{B}s)ds$ the corresponding generalized Iwatsuka potential (recall \eqref{beta}). Note that ${\mathcal B}(x)=B$ (resp. ${\mathcal B}(x)= -B$), when $x<B^{-1/2}q_-$ (resp. $x>B^{-1/2}q_+$).

If $q_-=q_+=0$ then we get a sharp interface and the magnetic potential is just $(0,-B|x|)$.

\begin{lemma}\label{lemlowlevel}
Let $\cA_{GIw}^{(B)}=(0,-\tilde{\beta})$ be the $(B, -B)$-generalized Iwatsuka potential described above. Then there exists a constant $c_0>0$ such that for all $B>0$ we have
\beq
\inf \sigma( H(A_{GIw}^{(B)})) \ge c_0 B.
\eeq
\end{lemma}

\begin{proof}
Performing the partial Fourier transform \eqref{fourier}, we end up with the analysis of $h(k)=-d^2/dx^2 + (k-\tilde{\beta}(x))^2$. Using dilations $(U\psi)(x)=\psi(x\sqrt{B})$, we see that $U^{-1} h(k) U =\tilde{h}(k,B)$ with
\begin{align}
\tilde{h}(k,B) &:=  -B \frac{d^2}{dx^2} + (k-\tilde{\beta}(x/\sqrt{B}))^2 \\
&= B \left( -\frac{d^2}{dx^2} + \left(\frac k{\sqrt{B}}-\sqrt{B}\int_0^{x/\sqrt{B}} \theta(s\sqrt{B})ds\right)^2 \right) \\
&= B \left(- \frac{d^2}{dx^2} + \left(\frac k{\sqrt{B}}-\int_0^{x} \theta(s)ds\right)^2 \right)
\\
& = B \tilde{h}(kB^{-1/2},1)
\end{align}
As a consequence, by the virtue of the direct decomposition \eqref{directint}, we get
\begin{align}
\inf \sigma( H(A_{GIw}^{(B)})) &= \inf_k (\inf\sigma(h(k))) = B \inf_k (\inf\sigma(\tilde{h}(k,1))) \\
&= B \inf \sigma( H(A_{GIw}^{(1)})).
\end{align}
\end{proof}

Let $B$ be large enough so that $I\subset]-\infty,c_0B]$, where $c_0$ comes from Lemma~\ref{lemlowlevel}. In particular, $B_-, B_+ < B$. Let $\cA_{\mathrm{Iw}}^{(L)}$ (resp. $\cA_{\mathrm{Iw}}^{(R)}$), be  a $(B,B_+)$ (resp. $(-B_-,-B)$), Iwatsuka potential. Then \eqref{sumrule2} yields
\begin{align}
 \sigma_e^{(I)}(H(A_{GIw}^{(-B_-,B_+)})) =&
\\
 \sigma_e^{(I)}(H(A_{Iw}^{(B,B_+)})) + \sigma_e^{(I)}(H(A_{Iw}^{(-B_-,-B)}))- \sigma_e^{(I)}(H(A_{GIw}^{(-B,B)}))
= & \\- n_+ - n_-,
\end{align}
recovering Corollary~\ref{corsnake}.


\section{Regularization and disordered systems}
\label{sectregul}
Quantum Hall effect actually deals with disordered systems, for
its famous plateaux are consequences of the existence of localized
states. As noticed in \cite{CG,EGS}, when adding a random
potential the definition of the edge conductance requires a
regularization to make sense. This regularization encodes the
localization properties of the disordered systems, killing
possible spurious currents.

Following \cite{CG}, a family $\lbrace J_r \rbrace _{r>0}$ will be called a regularization for an Hamiltonian $H$
and an interval $I$ if the following conditions hold true
\begin {align}
&(C1):\Vert J_r \Vert =1 ,\forall r>0 \mbox{ and } \forall \varphi\in E_{H}(I) L^2(\R^2) ,
\mbox{we have} \lim_{r\rightarrow\infty}J_r\varphi=\varphi\\
& (C2): g'(H)i[H,\chi]J_r \in\T_1,\forall r>0, \mbox{ and }
\mbox{there exists} \lim_{r\rightarrow \infty } \tr(g'(H)i[H,\chi]J_r) <  \infty.
\end {align}
For such a regularization we can define the regularized edge conductance by
\begin{equation}
 \sigma_e^{\mathrm{reg}, I}:=-2\pi\lim_{r\rightarrow\infty} \tr( g'(H)i[H,\chi]J_r)
\end{equation}
Consider now a pair ($H(\cA_{\mathrm{Iw}},a)$, $H(0,a)$). As an immediate consequence of Theorem~\ref{thm2},
if $J_R$ regularizes one operator of this pair then it regularizes the second one, and one has
\begin{equation}\label{rulereg}
     \sigma_e^{\mathrm{reg}, I}(H(\cA_{\mathrm{Iw}},a)) = n +  \sigma_e^{\mathrm{reg}, I}(H(0,a)).
\end{equation}
In particular, if we can show that
$\sigma_e^{\mathrm{reg}, I}(H(0,a))=0$,
for instance under some localization property, then the edge quantization for $H(\cA_{\mathrm{Iw}},a)$ satisfies,
for any $I\subset ](2n-1)B_-, (2n+1)B_-[\cap ]-\infty, B_+[$:
\begin{equation}
 \sigma_e^{\mathrm{reg}, I}(H(\cA_{\mathrm{Iw}},a))= - 2\pi\lim_{R\to\infty} \tr( g'(H(\cA_{\mathrm{Iw}},a))i[H(\cA_{\mathrm{Iw}},a),\chi]J_R)= n.
\label{eqregul}
\end{equation}
The value $n$ is of course in agreement with the value of the (bulk) Hall conductance, as argued by Halperin \cite{Ha}. Indeed, by extending \cite{GKS1} of \cite{GKS2} to random magnetic perturbations, the Hall conductance can be defined and  computed for Fermi energies lying in the localized states region, and shown to be equal  to the number of the highest Landau level below the Fermi level.

Several possible regularizations have been introduced in \cite{CG,CGH} within this context,
each of them based on a specific localization property of the interface random Hamiltonian, where the randomness is only located on the half plane where the energy barrier is not effective. These  properties are known to hold in the region of complete localization \cite{GKduke,GKjsp}. It is worth pointing out that this non ergodic random Hamiltonian is the relevant operator within our context where we deal with interface issues.

In particular, in some situations, it is possible to observe edge currents ``without edges", meaning edge currents created by an interface random potential, as shown in \cite{CG,CGH}. Playing with the sum rule it is actually possible to show the quantization of the regularized edge conductance for models considered in \cite{EKSS}, namely two different random electric potentials on the left and right half spaces, provided that the disorder difference is large. It can be extended to a high disorder electric potential and a small disorder magnetic potential. We cannot yet prove such a phenomenon for purely magnetic random  potentials.

Such regularizations are thus designed to study the interface problem, and compute directly the edge conductance. The equality with the bulk conductance is then a by-product of this computation if by other means the bulk conductance could be computed.

If the focuss is rather put on the equality-bulk edge, then a similar regularization to the ones above, but involving the localization properties of the $\Z^2$ ergodic bulk Hamiltonian is needed. Such an analysis is pulled through in \cite{EGS} where  the authors  are able to reconciliate a priori the edge and bulk points of view, showing that their regularized edge conductance and the Hall conductance match. It is very likely that such an analysis can be carried over to the context of the present paper. However it would require to extend the analysis of \cite{EGS} to the continuous setting and to random magnetic potentials.

To illustrate the discussion of this section,  let us consider the random magnetic field
$$
a_{\omega}(x,y)=
 \sum_{j=(j_1,j_2)\in\Z^2, \; j_1\ge 0} \omega_j
{\bf v}_j(x,y),
$$
with ${\bf v}_j(x,y)=(v_1(x-j_1),v_2(y-j_2))$, $v_1,v_2$ being two given $L^\infty$ compactly supported functions, and $\omega_j$ independent and identically
distributed random variables supported on $[-1,1]$, with common density
$\rho_\eta(s)=C_{\eta} \eta^{-1} \exp(-s \eta^{-1})
\chi_{[-1,1]}(s)$, $\eta>0$, where $C_\eta$ is such that $\int \rho_\eta(s) ds =1$. The support of $\rho_\eta$ is $[-1,1]$ for all $\eta>0$, but as $\eta$ goes to zero the disorder becomes weaker, for most $j$ the coupling $\omega_j$ will be almost zero. This model is the half-plane perturbation version  of the model considered in \cite{DGR2}.

We denote by $H_{B_-,\lambda,\omega}$ the corresponding random operator $(-i\nabla - A_0 - \lambda a_{\omega})^2$, where $A_0$  generates a constant magnetic field of strength $B_-$ in the perpendicular direction. By Lemma~\ref{lemgap}, for $\lambda$ small enough, the spectrum of $H_{B_-,\lambda,\omega}$ is
contained in disjoint intervals centered at the Landau levels. Thanks to the ergodicity in the $y$ direction, the spectrum is almost surely deterministic (see e.g. \cite[Theorem~2]{EKSS}, which can be extended to random perturbation of order $1$ considered here). In \cite{DGR2}, we show that for the $\Z^2$-ergodic version of $H_{B_-,\lambda,\omega}$, there exists $\lambda^\ast>0$ and for any $\lambda\in]0,\lambda^{\ast}]$ some $\eta^\ast(\lambda)$ such that for any $\eta\in]0,\eta^\ast(\lambda)]$, the full picture of localization as described in \cite{GKduke,GKjsp} is valid at the edge of the spectrum. The same analysis holds true for the half-plane version of the randomness  as well for the Wegner estimate of \cite{HK} holds the same (the same vector field can be used), and the initial condition is verified the same way uniformly for all boxes of the initial scale. This comes from the fact that within the region where the magnetic perturbation is zero, the localization holds for free at a given distance to the Landau levels.

The remaining issue is to make sure that the spectrum is not empty where localization can be proven. In \cite{DGR2} specific perturbations $a_{\omega}$ are explicitly constructed where for a given integer $J$, the $J^{\rm th}$ first Landau levels of the Landau Hamiltonian $H(B_-)$ split into non trivial intervals as $\lambda$ is turned on. It then follows from \cite[Theorem~2]{EKSS} (extended to magnetic perturbations) that these intervals are also contained in the spectrum of the corresponding $H_{B_-,\lambda,\omega}$. The reason for that result is that if $E$ belongs to the almost sure spectrum of the $\Z^2$-ergodic model, it is approximated, in term of a Weyl sequence, by eigenvalues of a large volume Hamiltonian with a specific configuration of the random variables; then almost every potential, even defined on the half space will exhibit somewhere the same pattern, thus creating the same eigenvalue.

\section{Appendix: trace estimates}
\label{sectAuxiliary}

Let ${\mathcal H}$ be a given separable Hilbert space. Denote by
${\mathcal B}$ the class of bounded linear operators with norm
$\|\cdot\|$, acting in ${\mathcal H}$, and by ${\mathcal T}_p$, $p
\in [1,\infty[$, the Schatten-von Neumann class of compact
operators acting in ${\mathcal H}$. We recall that ${\mathcal
T}_p$ is a Banach space with norm $\|T\|_p : = \left({\rm
tr}\,(T^*T)^{p/2}\right)^{1/p}$. In particular, in coherence with
our previous notations, ${\mathcal T}_1$ is the trace class, and
${\mathcal T}_2$ is the Hilbert-Schmidt class. The following lemma
contains some well-known properties of the Schatten-von Neumann
spaces, used systematically in the proofs of our results.

\begin{lemma} \label{gl1} {\rm \cite{Si}}
(i) Let $T \in {\mathcal T}_p$, $p \in [1,\infty[$. Then $T^* \in {\mathcal T}_p$
and we have
    \bel{g0}
    \|T\|_p = \|T^*\|_p.
    \ee
    (ii) Let $T \in {\mathcal T}_p$, $p \in [1,\infty[$, and $Q \in {\mathcal B}$.
    Then $TQ \in {\mathcal T}_p$, and we have
    \bel{g00}
    \|TQ\|_p \leq \|T\|_p \|Q\|.
    \ee
    (iii) Let $p_j \in [1,\infty[$, $j = 1\ldots, n$, $p \in [1,\infty[$, and $\sum_{j=1}^n p_j^{-1} = p^{-1}$.
    Assume that $T_j \in {\mathcal T}_{p_j}$,
    $j = 1\ldots, n$. Then $T : = T_1 \ldots T_n \in {\mathcal T}_p$, and we have
    \bel{g000}
     \|T\|_p \leq \prod_{j=1}^n \|T\|_{p_j}.
     \ee
     \end{lemma}
     \begin{lemma} \label{gl3} {\rm \cite{Si}}
     (i) Let $T \in {\mathcal T}_1$, $Q \in {\mathcal B}$. Then we have
     \bel{g60}
     {\rm tr}\,TQ = {\rm tr}\,QT.
     \ee
     (ii) Let $p \in [1,\infty[$, $q \in [1,\infty[$, $p^{-1} + q^{-1} = 1$. Assume that $T \in {\mathcal T}_p$, $Q \in {\mathcal T}_q$. Then \eqref{g60} holds true again.
     \end{lemma}
     Our next lemma contains a simple condition which guarantees the inclusion
     $T \in {\mathcal T}_p$ for operators of the form $T = f(x)g(-i\nabla)$.
     \begin{lemma} \label{gl2} {\rm \cite[Theorem 4.1]{Si}}
     Let $d \geq 1$, $p \in [2,\infty[$, $f,
     g \in L^p(\re^d)$. Set $T: = f(x)g(-i\nabla)$. Then we have $T \in {\mathcal T}_p$, and
     \bel{g04}
\|T\|_p \leq (2\pi)^{-d/p}  \|f\|_{L^{p}} \|g\|_{L^{p}}.
    \ee
\end{lemma}
Assume that
    \bel{g07}
    \beta \in L^{\infty}(\rd; {\mathbb C}^2), \quad
    {\rm div}\,\beta \in L^{\infty}(\rd).
    \ee
     Define the operator
    $$
    {\mathcal L}_{\beta}u : = \beta \cdot \nabla u, \quad u \in {\mathcal C}_0^{\infty}(\rd),
    $$
    and then close it in $L^2(\rd)$.
    \begin{proposition} \label{gp1}
    Let $A \in {\mathcal C}^1(\rd ; \rd)$, $z \in {\mathbb C}\setminus[0,\infty[$. Set  $R_A(z) : = (H(A) - z)^{-1}$. \\
    (i) Assume that $\alpha \in L^2(\rd)$. Then we have
    \bel{g1}
    \alpha R_A(z) \in {\mathcal T}_2, \quad   R_A(z) \alpha \in {\mathcal T}_2.
    \ee
    Moreover, there exists a constant $C_1$ independent of $z$,
    such that
    \bel{g42}
    \|\alpha R_A(z)\|_2 = \|R_A(z) \alpha\|_2 \leq C_1 C_0(z)
    \ee
    where
    \bel{g40}
    C_0(z) : =  \sup_{\lambda \in [0,\infty[} \frac{\lambda + 1}{|\lambda -
    z|}.
    \ee
    (ii) Assume that $\beta$ is compactly supported and satisfies \eqref{g07}.  Then we have
    \bel{g2}
    {\mathcal L}_\beta R_A(z) \in {\mathcal T}_4, \quad   R_A(z)  {\mathcal L}_\beta  \in {\mathcal T}_4.
    \ee
    Moreover, there exists a constant $C_2$ independent of $z$,
    such that
    \bel{g43}
    \|{\mathcal L}_\beta R_A(z)\|_4 \leq C_2 C_0(z), \quad \|R_A(z)  {\mathcal L}_\beta\|_4 \leq C_2
    C_0(z).
    \ee
    \end{proposition}
    \begin{proof}
    (i) By \eqref{g0} it suffices to prove only the first inclusion in \eqref{g1}, which  follows immediately from
    $$
    \|\alpha R_A(z)\|_2 = \|\alpha R_A(-1) (H(A) + 1) R_A(z)\|_2 \leq C_0(z) \|\alpha R_A(-1) \|_2 \leq
    $$
    \bel{g3}
    C_0(z) \|\alpha (-\Delta + 1)^{-1} \|_2 =
    \frac{C_0(z)}{2\pi} \left( \int_{\rd} |\alpha|^2 dx \int_{\rd} \frac{d\xi}{(|\xi|^2 + 1)^2}\right)^{1/2} < \infty.
    \ee
    Note that
    the second inequality is a special case of the diamagnetic inequality of Hilbert-Schmidt operators
    (see e.g. \cite[Theorem 2.13]{Si}), and the last equality just follows from the Parseval identity. \\
    (ii) Since we have
    $$
    (R_A(z) {\mathcal L}_\beta)^* = (-{\mathcal L}_{\overline{\beta}}-{\rm div}\,\overline{\beta}) R_A(\overline{z})
    $$
    and ${\rm div}\,\overline{\beta} R_A(\overline{z}) \in \mathcal{T}_2 \subset \mathcal{T}_4$ by \eqref{g1},
    again it suffices to check only the first inclusion in \eqref{g2}. As in the proof of \eqref{g1}
    we have
    $$
\|{\mathcal L}_\beta R_A(z)\|_4 \leq C_0(z) \|{\mathcal L}_\beta
R_A(-1)\|_4.
$$
    Further,
    \bel{g20}
    {\mathcal L}_\beta R_A(-1) = i\beta\cdot(-i\nabla -A)R_A(-1) + i\beta\cdot A R_A(-1),
    \ee
    and $i\beta\cdot A R_A(-1) \in \mathcal{T}_2 \subset \mathcal{T}_4$ by \eqref{g1}.
    Let $0 \leq \zeta_j \in {\mathcal C}_0^{\infty}(\rd)$, $j=0,1$, satisfy
    $\zeta_0 \beta = \beta$, $\zeta_1 \zeta_0 = \zeta_0$ on $\rd$.
    Since $[R_A(-1), \zeta_0] = - R_A(-1) [H(A), \zeta_0]
    R_A(-1)$, we have
    $$
    i\beta\cdot(-i\nabla -A)R_A(-1) =
    $$
    \bel{g21}
    i\beta\cdot(-i\nabla -A)R_A(-1)\zeta_0 -
    i\beta\cdot(-i\nabla -A)R_A(-1)\zeta_1 [H(A), \zeta_0] R(-1).
    \ee
    Note that the operator
    $$
    [H(A), \zeta_0] R_A(-1) = 2 \nabla \zeta_0 \cdot (-\nabla + iA)R_A(-1) - \Delta \zeta_0 R_A(-1)
    $$
    is  bounded. Therefore, it follows from  \eqref{g20}, \eqref{g21}, \eqref{g00}, that it suffices to
    check
    \bel{g22}
    \beta\cdot(-i\nabla -A)R_A(-1)\zeta \in {\mathcal T}_4
    \ee
    with $0 \leq \zeta \in {\mathcal C}_0^{\infty}(\rd)$. The mini-max principle implies
    \bel{g23}
    \|\beta\cdot(-i\nabla -A)R_A(-1)\zeta\|_4 \leq \|\beta\|_{L^{\infty}}
    \|R_A(-1)^{1/2}\zeta\|_4.
    \ee
    On the other hand,
    \bel{g24}
    \|R_A(-1)^{1/2}\zeta\|_4 = \|\zeta R_A(-1)^{1/2}\|_4
    \ee
    by \eqref{g0}. The diamagnetic inequality for ${\mathcal T}_4$-operators (see e.g. \cite[Theorem 2.13]{Si}) entails
    \bel{25}
    \|\zeta R_A(-1)^{1/2}\|_4 \leq \|\zeta (-\Delta + 1)^{-1/2}\|_4,
    \ee
    and by \eqref{g04} we obtain
    \bel{g26}
    \|\zeta (-\Delta + 1)^{-1/2}\|_4 \leq (2\pi)^{-1/2} \|\zeta\|_{L^4}
    \left(\int_{\rd} \frac{d\xi}{(|\xi|^2 + 1)^2}\right)^{1/4}.
    \ee
    Putting together \eqref{g23} - \eqref{g26}, we obtain \eqref{g22}, and hence \eqref{g2} - \eqref{g43}.
    \end{proof}

    \begin{remark} \label{gr2}
    If $\Im{z} \neq 0$, then the constant $C_0(z)$ defined in
    \eqref{g40} admits the estimate
    \bel{g41}
    C_0(z) \leq \frac{(\Re{z} + 1)_+}{|\Im{z}|} + 1.
    \ee
    \end{remark}

    \begin{corollary} \label{gf1}
    Let $A^{(j)} \in {\mathcal C}^1(\rd; \rd)$, $H_j : = H(A^{(j)})$,
    $R_j : = (H_j - z)^{-1}$,  $z \in {\mathbb C}\setminus[0,\infty[$, $j=1,2,3$. Assume that $\alpha \in L^2(\rd)$,
    $\beta$ satisfies \eqref{g07}, and $\alpha$ and $\beta$ are compactly supported. Then the operators
    \bel{g27}
    ({\mathcal L}_\beta + \alpha)R_j R_k R_l, \quad R_j({\mathcal L}_\beta + \alpha) R_k R_l,
    \quad R_j R_k ({\mathcal L}_\beta + \alpha) R_l, \quad R_j R_k R_l({\mathcal L}_\beta + \alpha)
    \ee
    with $j,k,l = 1,2,3$, are trace-class.
    \end{corollary}
    \begin{proof}
    By \eqref{g0} it suffices to consider only the first two operators in \eqref{g27}.
    Introduce three functions $0 \leq \zeta_j(\rd) \in {\mathcal C}_0^{\infty}(\rd)$, $j=0,1,2$, such that
    $\zeta_0 \alpha = \alpha$, $\zeta_0 \beta = \beta$, $\zeta_j \zeta_{j-1} = \zeta_{j-1}$, $j=1,2$. Note that
    $[R_j, \zeta_k] = - R_j [H_j, \zeta_k] R_j$, and
     \bel{g29}
     [H_j, \zeta_k] = 2 \nabla \zeta_k \cdot (-\nabla + iA^{(j)}) - \Delta \zeta_k
     \ee
     with $j=1,2,3$, and $k=0,1,2$. Then we have
     $$
     ({\mathcal L}_\beta + \alpha)R_j R_k R_l =
     $$
     $$
     ({\mathcal L}_\beta + \alpha)R_j \zeta_0 R_k \zeta_1 R_l  -
     ({\mathcal L}_\beta + \alpha)R_j \zeta_0 R_k [H_k, \zeta_1] R_k R_l -
     ({\mathcal L}_\beta + \alpha)R_j [H_j, \zeta_0]R_j  \zeta_1 R_k R_l +
     $$
     $$
     ({\mathcal L}_\beta + \alpha)R_j [H_j,\zeta_0] R_j [H_j, \zeta_1] R_j \zeta_2 R_k R_l -
     $$
     \bel{g30}
     ({\mathcal L}_\beta + \alpha)R_j [H_j,\zeta_0] R_j [H_j, \zeta_1] R_j [H_j,\zeta_2]R_j R_k R_l,
     \ee
     $$
     R_j({\mathcal L}_\beta + \alpha) R_k R_l =
     $$
     \bel{g31}
     R_j \zeta_0 ({\mathcal L}_\beta + \alpha) R_k \zeta_1 R_l -
     R_j \zeta_0 ({\mathcal L}_\beta + \alpha) R_k [H_k, \zeta_1] R_k R_l.
     \ee
     Taking into account  Proposition \ref{gp1} \eqref{g29},
     as well as \eqref{g000} with $p=1$ and \eqref{g00}, we find that \eqref{g30} and
\eqref{g31} imply that the operators in \eqref{g27} are
trace-class.
    \end{proof}

{\bf Acknowledgements}.
 N. D. and F. G. were supported in part by ANR 08 BLAN 0261. N. D. and G. R. were partially
supported by {\em N\'ucleo Cient\'ifico ICM} P07-027-F ``{\em
Mathematical Theory of Quantum and Classical Magnetic Systems"}.
   F. G. and G. R. were partially supported by the Chilean Scientific Foundation {\em
Fondecyt} under Grant 1090467. G. R. thanks for partial supports
the University of Cergy-Pontoise during his visit in 2008, and the
Bernoulli Center, EPFL, Switzerland, during his participation in
the Program ``{\em Spectral and Dynamical Properties of Quantum
Hamiltonians}" in
2010.\\

\end{document}